\documentclass[12pt, reqno]{amsart}   

\usepackage[utf8]{inputenc}           
\usepackage[T1]{fontenc}              
\usepackage{geometry}                 
\usepackage{amsmath, amssymb, amsthm} 

\usepackage{graphicx}                 
\usepackage[colorlinks=true, citecolor=blue, linkcolor=blue, urlcolor=blue]{hyperref}
\usepackage{cleveref}                 

\newtheorem{theorem}{Theorem}[section]
\newtheorem{lemma}[theorem]{Lemma}

\newtheorem{corollary}[theorem]{Corollary}

\newtheorem{assumption}[theorem]{Assumption}

\theoremstyle{definition}

\newtheorem{remark}[theorem]{Remark}

\numberwithin{equation}{section}      

\title[Volterra Integral Reduction]{Volterra Integral Reduction \\ for Boundary Diffusion Problems}

\author{Danila Shabalin}
\address{Lomonosov Moscow State University, Faculty of Mechanics and Mathematics, Moscow, Russia}
\email{danilshabalin2002@mail.ru}


\begin{document}

\begin{abstract}
This paper addresses a class of integral representations of the form
\begin{equation} \label{eq:target}
f(t,x)=g(t,x)+\int_0^t k(t,s)\, p(t-s,x,y)\, \partial_x f(s,y^+)\,ds, \qquad 0 \le t \le T,
\end{equation}
where $f$ is unknown, $p$ is the transition density of a diffusion process, and $g, k$ are prescribed functions. For an arbitrary diffusion process with sufficiently regular coefficients, we prove that this problem is equivalent to a Volterra integral equation of the second kind. This reduction provides a unified framework for both theoretical analysis and numerical approximation. An example of the implementation in the context of financial mathematics is presented.
\end{abstract}

\maketitle

\section{Introduction}
The expression \eqref{eq:target}, seemingly artificial, arises naturally in the solution of barrier option pricing \cite{guardasoni2018, mijatovic2010}, first-hitting-time problems \cite{detemple2026}, and American-style options \cite{lipton2020}, which are related to boundary Dirichlet and Stefan problems for second-order elliptic partial differential equations. The methods leading to these expressions include Green's formulas \cite{guardasoni2018}, local-time techniques \cite{detemple2026, mijatovic2010}, and the method of heat potentials \cite{lipton2020}.

Previously, resolving this problem required reducing it to a Volterra integral equation of the first kind with a weakly singular kernel, obtained by taking the limit $x \to y^+$ on both sides of \eqref{eq:target}. This yields
\begin{equation} \label{eq:VIE-first}
f(t,y)
=
g(t,y)
+
\int_0^t
k(t,s)\,
p(t-s,y,y)\,
\partial_x f(s,y^+)\,ds,
\end{equation}
where $f(t,y)$ is a known boundary value of the function $f(t,\cdot)$ at $y$ (we implicitly assume the existence of common limits for $f(t,\cdot)$, $g(t,\cdot)$, and $p(t,\cdot,y)$). Next, after solving equation \eqref{eq:VIE-first}, which is almost always done numerically, we substitute $\partial_x f$ into the initial expression \eqref{eq:target} and recover $f$. The formulation \eqref{eq:VIE-first} is ill-posed \cite[p.~33]{brunner2017} and presents certain challenges for numerical treatment; see \cite[Section~5]{detemple2026} for details. Therefore, we show that the first-kind Volterra equation can be avoided, and the problem can be directly transformed into a novel second-kind equation with a non-singular kernel, which is more suitable for numerical analysis.

To this end, in Section~\ref{sec:asympt}, we first establish the asymptotic properties of the transition density $p$ and its spatial derivative $\partial_x p$, which appear in the kernel \eqref{eq:target}. In Section~\ref{sec:VIE}, we then differentiate both sides of the resulting formula with respect to $x$ and take the limit $x \to y^+$, as before. Upon observing that one of the terms forms an approximation to the identity, this immediately yields the desired equation. Finally, in Section~\ref{sec:App}, we provide an example illustrating the applications of the new equation.

This work is a continuation of Section~6 in the previous paper \cite{detemple2026}, in which such a second-kind equation was obtained for the particular case of geometric Brownian motion in the context of the first-passage-time problem. In the present work, we generalize the approach to arbitrary diffusion meeting certain regularity conditions.

We consider a time-homogeneous Itô diffusion $X_t =\left(X_t\right)_{t\geq 0}$ defined on a filtered probability space
$(\Omega, \mathcal{F}, \mathbb{F} =(\mathcal{F}_t)_{t\geq 0}, \mathbb{P})$, where the filtration satisfies the usual conditions, and $W_t = \left(W_t\right)_{t\geq 0}$ is an $\mathbb{F}$-standard Brownian motion. The process $X_t$ is given by the stochastic differential equation
\begin{equation}\label{eq:SDE}
dX_t = b(X_t)\,dt + \sigma(X_t)\,dW_t,\qquad X_0=x\in\mathbb{R},
\end{equation}
where drift $b$ and diffusion $\sigma$ are Borel-measurable functions. In this setting, $p(t,x,y)$ denotes the transition density function of the Markov process $X_t$, defined by
$\mathbb{P}(X_t \in A \mid X_0 = x)
= \int_A p(t,x,y)\,dy$, for any Borel set $A \subset \mathbb{R}$. Throughout the paper, we consider the following global assumptions to ensure the existence of a unique solution to \eqref{eq:SDE} and the smoothness of the transition density $p$, along with the existence and uniqueness of solutions to \eqref{eq:VIE-first} and the resulting integral equation below.
\begin{assumption} \label{ass:smooth}
$\sigma \in C^4(\mathbb{R})$ and $b \in C^3(\mathbb{R})$,  and have at most linear growth.
\end{assumption}
\begin{assumption} 
$\sigma > 0$ strictly on $\mathbb{R}$.
\end{assumption}
\begin{assumption} \label{ass:func}
For a terminal time $T>0$: $g(t,x)\in C^{0,1}([0,T]\times\mathbb{R})$, $g(0,y)=0$, and $k(t,s)\in C([0,T] \times [0,T])$, with $k(t,t)\neq 0$ and $k(t,t)\neq -\sigma^2(y)$ for every $t\in[0,T]$.
\end{assumption}
We will also use the following notation for uniform asymptotic estimates on a compact set $K\subset\mathbb{R}$. We write
$f(t,x,y)=O_K(g(t))$ as $t\to0^+$ if there exists a constant $C_K>0$, depending only on $K$, such that
$|f(t,x,y)|\le C_K |g(t)|$ for all $x,y\in K$ and all sufficiently small $t>0$. Hereafter, $O_K(\cdot)$ is understood in this uniform sense, and the limit $t\to0^+$ will sometimes be omitted.
\begin{remark}
All the arguments remain valid for an inhomogeneous diffusion $X_t$ and, consequently, for its transition density $p$, as well as for a smoothly time-dependent boundary $y(t)$ in place of the flat (lower) boundary $y$. For an upper boundary $y$, we take the left-hand limit $x \to y^{-}$.
\end{remark}

\section{Short-time asymptotics} \label{sec:asympt}
Short-time asymptotics is a canonical problem that has been extensively studied in the literature \cite[Section~5]{hsu2002}. One of the first results here is the logarithmic heat kernel relation on a complete Riemannian manifold, due to Varadhan \cite{varadhan1967}. Then, using a geometric approach, Molchanov \cite{molchanov1975} derived a series representation, while Watanabe \cite{watanabe1987} developed a corresponding asymptotic expansion using Malliavin calculus on Wiener spaces. Finally, Ben Arous \cite{benarous1988,benarous1989} developed
these ideas by employing the Laplace method for hypoelliptic heat kernels and their derivatives, with coefficients determined recursively by the geometry of the manifold and variations of the action functional. Nevertheless, for our modest purposes, only the one-dimensional case is needed, as, for example, in \cite{gatheral2012}. The main result of this section can be formulated as follows.
\begin{theorem} \label{th:asymptotic}
For every compact set $K\subset\mathbb R$, the spatial derivative of the
transition density $p(t,x,y)$ admits, as $t\to0^+$, the asymptotic expansion
\begin{equation} \label{eq:density_relation}
\partial_x p(t,x,y)
=
e^{-d^2(x,y)/(2t)}\!\left[
c_0(x,y)\,t^{-3/2}
+
c_1(x,y)\,t^{-1/2}
+
O_K\!\left(
t^{1/2}
\right)
\right],
\end{equation}
uniformly for $(x,y)\in K\times K$, in which $c_0(x,y)$ and $c_1(x,y)$ are explicit coefficients given in Remark~\ref{rem:coefficients}, while $d(x,y)$ is a suitable distance defined in Lemma~\ref{lemma:BB}.
\end{theorem}
To start, we formulate two preliminary lemmas. The first provides a Brownian bridge representation of the transition density for diffusion, which can be found in \cite{wang2015}. For completeness, we briefly recall a proof.
\begin{lemma} \label{lemma:BB}
The transition density $p(t,x,y)$ of $X_t$ admits the representation
\begin{equation}
p(t,x,y)
=
\frac{1}{\sigma(y)}
p_0(t,u,v)\,
e^{B(v)-B(u)}
\mathbb{E}^{t,v}_{0,u}
\left[
\exp\!\left(
-\int_0^t V(W_s)\,ds
\right)
\right],
\end{equation}
where $u=\psi(x)$ and $v=\psi(y)$, while
$p_0(t,u,v)=\frac{1}{\sqrt{2\pi t}}\exp{\left[\frac{(v-u)^2}{2t}\right]}$ is the Gaussian transition density. The expectation is taken with respect to a Brownian bridge from $u$ at
time $0$ to $v$ at time $t$. The functions $\psi(\cdot)$, $V(\cdot)$, and $B(\cdot)$ are introduced in the course of the proof.
\end{lemma}
\begin{proof}
Define the Lamperti transform $\psi(x)=\int_{x_0}^{x} \frac{1}{\sigma(z)}\, dz$. This is a strictly increasing $C^5$-diffeomorphism from $\mathbb R$ onto the open interval
$I:=\psi(\mathbb R)$. Then, the distance is given by $d(x,y) = |\psi(y) - \psi(x)|$. Setting $Y_t=\psi(X_t)$, by It\^o's formula,
\begin{equation}
dY_t=\beta(Y_t)\,dt+dW_t,
\qquad
Y_0=u,
\end{equation}
where the drift $\beta$ is defined by
\begin{equation}
\beta(\psi(x))
=
\frac{b(x)}{\sigma(x)}
-\frac{1}{2}\sigma'(x).
\end{equation}
As $b\in C^3(\mathbb R)$, $\sigma\in C^4(\mathbb R)$, and
$\sigma>0$, we have $\beta\in C^3(I)$. The transition densities $p(t,x,y)$ of $X_t$ and $q(t,u,v)$ of $Y_t$ are related by
\begin{equation}
\label{eq:density-relation}
p(t,x,y)
=
\frac{1}{\sigma(y)}
q(t,\psi(x),\psi(y)).
\end{equation}
Let $\mathbb Q_u$ denote the Wiener measure on $C([0,t];\mathbb{R})$
corresponding to a Brownian motion starting from $u$, and let $\mathbb P_u$ denote the law of the process $Y_t$. To apply Girsanov's theorem, we need the stochastic exponential below to be a true martingale; it suffices to assume the Novikov condition. In this case,
\begin{equation}
\frac{d \mathbb P_u}{d \mathbb Q_u}\bigg|_{\mathcal{F}_t}
=
\exp\!\left(
\int_0^t \beta(W_s)\,dW_s
-\frac{1}{2}\int_0^t \beta^2(W_s)\,ds
\right).
\end{equation}
Introduce the anti-derivative $B(u) = \int_{u_0}^u \beta(z)\,dz$ and define the potential
\begin{equation}
V(z)
=
\frac{1}{2}
\left(
    \beta^2(z) + \beta'(z)
\right).
\end{equation}
It follows that $V\in C^2(I)$. Applying It\^o's formula to $B(W_t)$ yields
\begin{equation}
\frac{d \mathbb P_u}{d \mathbb Q_u}\bigg|_{\mathcal{F}_t}
=
\exp\!\left(
B(W_t)-B(u)
-\int_0^t V(W_s)\,ds
\right).
\end{equation}
Hence, conditioning on $W_t=v$, we arrive at
\begin{equation}
\label{eq:qbridge}
q(t,u,v)
=
p_0(t,u,v)\,
e^{B(v)-B(u)}
\mathbb{E}^{t,v}_{0,u}
\left[
\exp\!\left(
-\int_0^t V(W_s)\,ds
\right)
\right],
\end{equation}
where the expectation $\mathbb{E}^{t,v}_{0,u} \left[ \cdot\right]$ is with respect to the Brownian bridge from $u$ at time 
$0$ to $v$ at time $t$. Inserting \eqref{eq:qbridge} into \eqref{eq:density-relation} completes the proof.
\end{proof}
We next establish the short-time expansion of the Brownian bridge expectation in \eqref{eq:qbridge}, including its derivative with respect to the initial point $x$. Considering that we need to control the contributions from Brownian bridge trajectories that leave compact subsets, we impose the following additional assumption.
\begin{assumption} \label{ass:potential}
Assume that $V$ is polynomial bounded along with its first two derivatives, and bounded below; i.e. there exist constants $C_V,m>0$ such that $|V(z)|+|V'(z)|+|V''(z)|\le C_V(1+|z|^m)$ and $V(z)\ge -C_V$ for all $z\in I$.
\end{assumption}
We also define the averaged potential
$
\overline{V}(u,v)
=
\int_0^1
V\!\big((1-\lambda)u+\lambda v\big)\,d\lambda$.
\begin{lemma}
\label{lemma:bridge-expansion}
For every compact set $K\subset\mathbb R$, we set $K':=\psi(K)$. For $u,v\in K'$, introduce
\begin{equation}
\label{eq:F-def}
F(t,u,v)
:=
\mathbb E^{t,v}_{0,u}
\left[
\exp\left(
-\int_0^tV(W_s)\,ds
\right)
\right].
\end{equation}
Then, as $t\to0^+$,
\begin{align}
F(t,u,v)
&=
1-t\overline{V}(u,v)
+
O_{K'}(t^2),
\label{eq:F-expansion}
\\
\partial_uF(t,u,v)
&=
-t\,\partial_u\overline{V}(u,v)
+
O_{K'}(t^{3/2}),
\label{eq:Fu-expansion}
\end{align}
uniformly for $(u,v)\in K'\times K'$.
\end{lemma}
\begin{proof}
Rescale time by setting $s=\lambda t$, with $\lambda\in[0,1]$. A Brownian bridge from
$u$ to $v$ over $[0,t]$ can be represented as
\begin{equation}
W_{\lambda t}
=
c_{u,v}(\lambda)+\sqrt{t}\,\eta_\lambda,
\end{equation}
where
$c_{u,v}(\lambda):=(1-\lambda)u+\lambda v$,
and $\eta_\lambda$ is a standard Brownian bridge on $[0,1]$ (centered Gaussian,
$\mathbb{E}[\eta_\lambda]=0$, $\mathbb{E}[\eta_\lambda^2] = \lambda(1-\lambda)$).
Consequently,
\begin{equation}
\label{eq:It-def}
I_t(u,v)
:=
\int_0^t V(W_s)\,ds
=
t\int_0^1
V\!\big(
c_{u,v}(\lambda)+\sqrt{t}\,\eta_\lambda
\big)\,d\lambda.
\end{equation}
Let $K'_{\mathrm{ch}}:=\operatorname{conv}(K')$. Choose $\delta>0$ such that the compact neighborhood
\begin{equation}
K'_\delta
:=
\left\{
z\in\mathbb R:
\operatorname{dist}(z,K'_{\mathrm{ch}})\le\delta
\right\}
\end{equation}
is contained in $I$. Since $K'$ is compact, so is
$K'_{\mathrm{ch}}$, and hence $K'_\delta$ is compact. In particular,
\begin{equation}
M_j
:=
\sup_{z\in K'_\delta}|V^{(j)}(z)|
<\infty,
\qquad j=0,1,2.
\end{equation}
Although $u,v\in K'$, the Brownian bridge may leave
$K'_{\mathrm{ch}}$. Define the event
\begin{equation}
\Omega_t
:=
\left\{
\sup_{0\le \lambda\le1}
\sqrt{t}\,|\eta_\lambda|
\le\delta
\right\}.
\end{equation}
On $\Omega_t$, $ c_{u,v}(\lambda)+\sqrt{t}\,\eta_\lambda\in K'_\delta$, $\lambda\in[0,1]$. A standard Gaussian maximal estimate for the Brownian bridge gives
\begin{equation}
\label{eq:bridge-tail}
\mathbb P(\Omega_t^c)
=
\mathbb P\left(
\sup_{0\le \lambda\le1}|\eta_\lambda|
>
\frac{\delta}{\sqrt{t}}
\right)
\le
2\exp\left(
-\frac{2\delta^2}{t}
\right).
\end{equation}
On $\Omega_t$, Taylor's formula yields
\begin{equation}
V(c_{u,v}(\lambda)+\sqrt{t}\,\eta_\lambda)
=
V(c_{u,v}(\lambda))
+
V'(c_{u,v}(\lambda))\sqrt{t}\,\eta_\lambda
+
R_t^1(\lambda),
\end{equation}
where $|R_t^1(\lambda)|
\le
M_2\,t\,\eta_\lambda^2$. We estimate the expectation of
$V(c_{u,v}(\lambda)+\sqrt{t}\,\eta_\lambda)$. Since
$\mathbb E[\eta_\lambda]=0$, we may write
\begin{align}
&\mathbb E\left[
V(c_{u,v}(\lambda)+\sqrt{t}\,\eta_\lambda)
\right]
-
V(c_{u,v}(\lambda))
\nonumber\\
&\quad=
\mathbb E\left[
\left(
V(c_{u,v}(\lambda)+\sqrt{t}\,\eta_\lambda)
-
V(c_{u,v}(\lambda))
-
V'(c_{u,v}(\lambda))\sqrt{t}\,\eta_\lambda
\right)\!
\mathbf 1_{\Omega_t}
\right]
\nonumber\\
&\qquad+
\mathbb E\left[
\left(
V(c_{u,v}(\lambda)+\sqrt{t}\,\eta_\lambda)
-
V(c_{u,v}(\lambda))
-
V'(c_{u,v}(\lambda))\sqrt{t}\,\eta_\lambda
\right)\!
\mathbf 1_{\Omega_t^c}
\right].
\end{align}
The contribution from $\Omega_t$ is bounded by $M_2t$, while
the impact of $\Omega_t^c$ is exponentially small by the polynomial growth assumption on $V$ and $V'$ and the Gaussian tail estimates for the Brownian bridge. Combining these estimates gives
\begin{equation}
\label{eq:V-expectation}
\mathbb E\left[
V(c_{u,v}(\lambda)+\sqrt{t}\,\eta_\lambda)
\right]
=
V(c_{u,v}(\lambda))
+
O_{K'}(t),
\end{equation}
uniformly for $\lambda\in[0,1]$. Integrating \eqref{eq:V-expectation} with respect to $\lambda$ and using
\eqref{eq:It-def}, we find
\begin{equation}
\label{eq:EI-expansion}
\mathbb E[I_t(u,v)]
=
t\int_0^1V(c_{u,v}(\lambda))\,dr
+
O_{K'}(t^2)
=
t\overline{V}(u,v)
+
O_{K'}(t^2).
\end{equation}
Next, by Assumption~\ref{ass:potential}, $V$ is bounded below,
and therefore
\begin{equation}
\label{eq:exp-bound}
\exp(-I_t(u,v))
\le
e^{C_Vt}.
\end{equation}
Moreover, the polynomial growth assumption on $V$, together with the Gaussian moment estimates for the Brownian bridge, yields
\begin{equation}
\label{eq:It-second-moment}
\sup_{(u,v)\in K'\times K'}
\mathbb E\left[
I_t^2(u,v)
\right]
=
O_{K'}(t^2).
\end{equation}
Using
\begin{equation}
|e^{-z}-1+z|
\le
\frac12e^{|z|}z^2,
\end{equation}
together with \eqref{eq:It-second-moment} and
\eqref{eq:exp-bound}, we deduce
\begin{equation}
F(t,u,v)
=
1-\mathbb E[I_t(u,v)]
+
O_{K'}(t^2).
\end{equation}
Bringing this together with \eqref{eq:EI-expansion} proves
\eqref{eq:F-expansion}. We proceed to derive the expansion for the derivative with respect to $u$. Differentiating \eqref{eq:It-def}, we have
\begin{equation}
\label{eq:It-derivative}
\partial_u I_t(u,v)
=
t\int_0^1
(1-\lambda)
V'\!\left(
c_{u,v}(\lambda)+\sqrt{t}\,\eta_\lambda
\right)\,d\lambda.
\end{equation}
Since $V''$ is bounded on $K'_\delta$, Taylor's formula gives, on
$\Omega_t$,
\begin{equation}
V'\!\big(
c_{u,v}(\lambda)+\sqrt{t}\,\eta_\lambda
\big)
=
V'(c_{u,v}(\lambda))
+
R_t^2(\lambda),
\end{equation}
where $|R_t^2(\lambda)|
\le
M_2\sqrt{t}\,|\eta_\lambda|$. We split again the expectation into
$\Omega_t$ and $\Omega_t^c$ as above. The contribution from $\Omega_t$
is bounded by $M_2\sqrt{t}\,\mathbb E|\eta_\lambda|$, while the contribution from $\Omega_t^c$ is exponentially small by the
polynomial growth assumption on $V'$ and the Gaussian tail estimate
\eqref{eq:bridge-tail}. Hence
\begin{equation}
\label{eq:EIt-derivative}
\mathbb E[\partial_u I_t(u,v)]
=
t\int_0^1
(1-r)V'(c_{u,v}(\lambda))\,dr
+
O_{K'}(t^{3/2}).
\end{equation}
Differentiating the averaged potential, this gives
\begin{equation}
\partial_u\overline{V}(u,v)
=
\int_0^1
(1-r)V'(c_{u,v}(\lambda))\,dr.
\end{equation}
Consequently,
\begin{equation}
\label{eq:EIt-derivative-final}
\mathbb E[\partial_u I_t(u,v)]
=
t\,\partial_u\overline{V}(u,v)
+
O_{K'}(t^{3/2}).
\end{equation}
The derivative may be passed under the expectation in
\eqref{eq:F-def}. Indeed, by \eqref{eq:It-derivative} and the polynomial
growth assumption on $V'$, the random variables
$e^{-I_t}\partial_uI_t$ are uniformly integrable for sufficiently small $t$. Thus, by dominated convergence,
\begin{equation}
\partial_uF(t,u,v)
=
-\mathbb E\left[
e^{-I_t(u,v)}
\partial_uI_t(u,v)
\right].
\end{equation}
Furthermore, using $e^{-I_t}=1+O_{K'}(I_t)$ and the estimates
$ \mathbb E[I_t^2]=O_{K'}(t^2)$, $\mathbb E[|\partial_u I_t|]=O_{K'}(t)$,
we show
\begin{equation}
\mathbb E\left[
(e^{-I_t}-1)\partial_u I_t
\right]
=
O_{K'}(t^2).
\end{equation}
Accordingly, by \eqref{eq:EIt-derivative-final},
\begin{equation}
\partial_uF(t,u,v)
=
-t\,\partial_u\overline V(u,v)
+
O_{K'}(t^{3/2}), \qquad t \to 0^+,
\end{equation}
uniformly for $(u,v)\in K'\times K'$. This proves
\eqref{eq:Fu-expansion}.
\end{proof}

\begin{corollary}
\label{cor:density-expansion}
For every compact set $K\subset\mathbb R$, the transition density $p(t,x,y)$ admits
the asymptotic expansion
\begin{equation}
\label{eq:density-expansion}
p(t,x,y)
=
\frac{1}{\sigma(y)\sqrt{2\pi t}}
\exp\left(
-\frac{(\psi(y)-\psi(x))^2}{2t}
+
B(\psi(y))-B(\psi(x))
\right)
\left[
1-t\,\overline{V}(\psi(x),\psi(y))
+
O_K(t^2)
\right],
\end{equation}
as $t\to0^+$, uniformly for $(x,y)\in K\times K$.
\end{corollary}
\begin{proof}
Set $u=\psi(x)$, $v=\psi(y)$. Thus $(u,v)\in K'\times K'$, where $K'=\psi(K)$. Substituting
\eqref{eq:F-expansion} into \eqref{eq:qbridge}, we obtain
\begin{equation}
q(t,u,v)
=
\frac{1}{\sqrt{2\pi t}}
\exp\left(
-\frac{(v-u)^2}{2t}
+
B(v)-B(u)
\right)
\left[
1-t\,\overline{V}(u,v)
+
O_{K'}(t^2)
\right].
\end{equation}
Using the relation for $p(t,x,y)$ and the fact that $\psi(K)=K'$, we conclude \eqref{eq:density-expansion}.
\end{proof}

We are now ready to prove Theorem~\ref{th:asymptotic}.

\begin{proof}[Proof of Theorem~\ref{th:asymptotic}]
Fix a compact set $K\subset\mathbb R$ and set $K':=\psi(K)$. Let
$u=\psi(x)$, $v=\psi(y)$. Thus $(u,v)\in K'\times K'$. Since $y$ is fixed with respect to the differentiation in $x$, we have, by \eqref{eq:density-relation},
\begin{equation}
\label{eq:der}
\partial_x p(t,x,y)
=
\frac{1}{\sigma(x)\sigma(y)}
\partial_u q(t,u,v)
\bigg|_{u=\psi(x),\,v=\psi(y)}.
\end{equation}
According to \eqref{eq:qbridge},
\begin{equation}
\label{eq:q-factorization}
q(t,u,v)
=
\frac{1}{\sqrt{2\pi t}}
\exp\left(
-\frac{(v-u)^2}{2t}
+B(v)-B(u)
\right)
F(t,u,v),
\end{equation}
where, by Lemma~\ref{lemma:bridge-expansion}, $F(t,u,v)
=
1-t\,\overline{V}(u,v)
+
O_{K'}(t^2)
$. It is convenient to differentiate the logarithm of $q(t,u,v)$. From
\eqref{eq:q-factorization},
\begin{equation}
\log q(t,u,v)
=
-\frac12\log(2\pi t)
-\frac{(v-u)^2}{2t}
+B(v)-B(u)
+\log F(t,u,v).
\end{equation}
Hence,
\begin{equation}
\partial_u\log q(t,u,v)
=
\frac{v-u}{t}
-\beta(u)
+
\frac{\partial_uF(t,u,v)}{F(t,u,v)}.
\end{equation}
Because $F^{-1}(t,u,v)=1+O_{K'}(t)$, uniformly on compact sets, combining this with \eqref{eq:Fu-expansion}, leads to
\begin{equation}
\frac{\partial_uF(t,u,v)}{F(t,u,v)}
=
-t\,\partial_u\overline{V}(u,v)
+
O_{K'}(t^{3/2}).
\end{equation}
Therefore,
\begin{equation}
\label{eq:log-q-final}
\partial_u\log q(t,u,v)
=
\frac{v-u}{t}
-\beta(u)
-t\,\partial_u\overline{V}(u,v)
+
O_{K'}(t^{3/2}).
\end{equation}
Multiplying \eqref{eq:log-q-final} by $q(t,u,v)$ and using
\eqref{eq:q-factorization}, we obtain
\begin{equation}
    \partial_u q(t,u,v)
    =
    \frac{1}{\sqrt{2\pi t}}
    \exp\!\left(
        -\frac{(v-u)^2}{2t}
        + B(v)-B(u)
    \right)\!
    H(t,u,v),
\end{equation}
where
\begin{equation}
\label{eq:H-product}
H(t,u,v)
=
\left[
    1-t\,\overline{V}(u,v)+O_{K'}(t^2)
\right]\!
\left[
    \frac{v-u}{t}
    -\beta(u)
    -t\,\partial_u \overline{V}(u,v)
    +O_{K'}(t^{3/2})
\right].
\end{equation}
Consequently, expanding the product in \eqref{eq:H-product} gives
\begin{equation}
\begin{aligned}
    \partial_u q(t,u,v)
    =&
    \frac{1}{\sqrt{2\pi}}
    \exp\!\left(
        -\frac{(v-u)^2}{2t}
        +B(v)-B(u)
    \right)
    \\
    &\times
    \left[
        \frac{v-u}{t^{3/2}}
        -
        \frac{
            \beta(u)+(v-u)\overline{V}(u,v)
        }{t^{1/2}}
        +O_{K'}(t^{1/2})
    \right].
\end{aligned}
\end{equation}
Finally, substituting $u=\psi(x)$, $v=\psi(y)$ and $d(x,y)=|\psi(y)-\psi(x)|$ into \eqref{eq:der} yields the desired expansion, which completes the proof of the theorem.
\end{proof}
\begin{remark}\label{rem:coefficients}
The coefficients in Theorem~\ref{th:asymptotic} are given by
\begin{align}
    c_0(x,y)
    &=
    \frac{\psi(y)-\psi(x)}
    {\sigma(x)\sigma(y)\sqrt{2\pi}}
    \exp\!\left(
        B(\psi(y))-B(\psi(x))
    \right),
    \\
    c_1(x,y)
    &=
    -\frac{
        \beta(\psi(x))
        +(\psi(y)-\psi(x))
        \overline V(\psi(x),\psi(y))
    }
    {\sigma(x)\sigma(y)\sqrt{2\pi}}
    \exp\!\left(
        B(\psi(y))-B(\psi(x))
    \right).
\end{align}
\end{remark}
While the coefficients and the corresponding expansions could also be recovered from the classical results discussed above, that route would likely be no more efficient than the present method. What matters is that we have established the uniform asymptotics needed for the subsequent analysis.

\section{Boundary limits and Volterra equations} \label{sec:VIE}
Returning to the integral representation \eqref{eq:target}, we differentiate it with respect to the spatial variable $x$. Since $k(t,s)$ is assumed to be smooth and, for $x\neq y$, $\partial_x p(t-s,x,y)$ admits an integrable majorant in $s$ near $s=t$, as given by \eqref{th:asymptotic}, we have
\begin{equation}
\partial_x f(t,x)
=
\partial_x g(t,x)
+
\int_0^t
k(t,s)\,
\partial_x p(t-s,x,y)\,
\partial_x f(s,y^+)\,
ds.
\label{eq:dxV}
\end{equation}
Our goal is to analyze the limit $x\to y^+$.
Plugging the short-time expansion \eqref{eq:density_relation} into \eqref{eq:dxV}, we decompose the differentiated kernel as
\begin{equation}
\partial_x p(\tau,x,y)
=
k_0(\tau,x,y)
+
k_1(\tau,x,y)
+
r_0(\tau,x,y), \qquad \tau=t-s,
\end{equation}
in which the singular behavior near the diagonal $x=y$, as $\tau$ approaches zero, is captured by $k_0$ and $k_1$, while the continuous $r_0$ remains regular in a neighborhood of the diagonal. Note that this is the only special point of $p$ that is shared by $k_0$ and $k_1$. Introducing the notation $\ell=d(x,y)$, the leading singular contribution can be written as
\begin{equation}
k_0(\tau,x,y)
=
-\frac{\exp\!\left[
        B(\psi(y))-B(\psi(x))
    \right]}
    {\sigma(x)\sigma(y)}
\,\rho_\ell(\tau),
\end{equation}
where
\begin{equation}
\rho_\ell(\tau)
=
\frac{\ell}
{\sqrt{2\pi}\,\tau^{3/2}}
\exp\!\left(
-\frac{\ell^2}{2\tau}
\right).
\end{equation}
The family $\{\rho_\ell\}_{\ell>0}$ is the classical Le\'vy density and satisfies
\begin{equation}
\int_0^\infty \rho_\ell(\tau)\,d\tau=1.
\end{equation}
Moreover, as $\ell\to0^+$, it forms an approximation to the identity on $[0,\infty)$. Thus, for every $h\in C[0,t]$,
\begin{equation}
\lim_{\ell\to0^+}
\int_0^{t}
h(\tau)\,
\rho_\ell(\tau)\,
d\tau
=
h(0).
\label{eq:approx_identity}
\end{equation}
Applying \eqref{eq:approx_identity} with
\begin{equation}
h(\tau)
=
-\frac{\exp\!\left(
        B(\psi(y))-B(\psi(x))
    \right)}
    {\sigma(x)\sigma(y)}
\,k(t,t-\tau)\,
\partial_x f(t-\tau,y^+),
\end{equation}
we obtain
\begin{equation}
\lim_{x\to y^+}
\int_0^t
k(t,s)\,
k_0(t-s,x,y)\,
\partial_x f(s,y^+)\,
ds
=
-\frac{k(t,t)}{\sigma^2(y)}
\,\partial_x f(t,y^+).
\label{eq:jump}
\end{equation}
Hence, the first term $k_0$ produces a jump contribution concentrated on the diagonal. We next consider the second term $k_1$. Since $(t-s)^{-1/2}\in L^1(0,t)$ for every $t>0$, the corresponding kernel remains integrable near $s=t$.  Therefore, by the dominated convergence theorem,
\begin{equation}
\lim_{x\to y^+}
\int_0^t
k(t,s)\,
k_1(t-s,x,y)\,
\partial_x f(s,y^+)\,
ds
=
-\frac{1}{\sqrt{2\pi}}
\int_0^t
\frac{k(t,s)\,\beta(\psi(y))}
{\sigma^2(y)(t-s)^{1/2}}\,
\partial_x f(s,y^+)\,ds.
\label{eq:K1_limit}
\end{equation}
Thus, $k_1$ contributes to the weakly singular part of the Volterra kernel. Finally, the remainder term $r_0$ is less singular. Indeed, $r_0(\tau,y,y)=O_K(\tau^{1/2})$, uniformly on compact sets $K$, and is integrable on $[0,t]$. This yields,
\begin{equation}
\lim_{x\to y^+}
\int_0^t
k(t,s)\,
r_0(t-s,x,y)\,
\partial_x f(s,y^+)\,
ds
=
\int_0^t
k(t,s)\,
r_0(t-s,y,y)\,
\partial_x f(s,y^+)\,
ds.
\label{eq:R_limit}
\end{equation}
At this point, combining \eqref{eq:jump}, \eqref{eq:K1_limit}, and \eqref{eq:R_limit} with \eqref{eq:dxV} and taking the limit as $x\to y^+$ provides the key theorem of our analysis.
\begin{theorem}\label{th:VIE-second}
Under Assumptions~\ref{ass:smooth}--\ref{ass:func} and~\ref{ass:potential}, 
for each fixed $y$, the function $t\mapsto \partial_x f(t,y^+)$ 
is the unique solution in $C[0,T]$ of the following Volterra integral equation 
of the second kind with a weakly singular kernel
\begin{equation}
\label{eq:VIE-second}
\alpha(t,y)\,\partial_x f(t,y^+)
=
\partial_x g(t,y)
+
\int_0^t
k(t,s)\,
\frac{\gamma_0(t-s,y)}
{(t-s)^{1/2}}\,
\partial_x f(s,y^+)\,ds,
\end{equation}
where
\begin{equation}
\alpha(t,y) = 
\frac{\sigma^2(y)+k(t,t)}
{\sigma^2(y)}, \qquad
\gamma_0(\tau,y) =
\tau^{1/2} r_0(\tau,y,y)
-
\frac{\beta(\psi(y))}
{\sqrt{2\pi}\,\sigma^2(y)},
\end{equation}
and $r_0(\tau,y,y)$ is a continuous implicit function satisfying 
$r_0(\tau,y,y)=O_K(\tau^{1/2})$ uniformly on compact sets.
\end{theorem}
\begin{proof}
It remains only to establish the uniqueness of the continuous solution, which follows from Theorem~1.3.5 in~\cite{brunner2017}, in view of Assumption~\ref{ass:func}.
\end{proof}
\begin{remark} \label{rem:conditions}
The conditions in Assumption~\ref{ass:func} can be relaxed. If $g(0,y)\neq 0$, then the solution may become unbounded as $t\to0^+$, belonging to $C(0,T]$ and exhibiting a singularity of order $t^{-\kappa}$ for some $\kappa>0$ near the origin. If $\sigma^2(y)+k(t,t)$ vanishes on a proper subset of $[0,T]$, the consequent equation falls into the class of Volterra integral equations of the third kind; see Section~1.6 of~\cite{brunner2017}. Each modification relies on its own theory, which does not concern us here.
\end{remark}
Now, we proceed further and remove the singularity in~\eqref{eq:VIE-second}. We observe that, as a straightforward consequence of Corollary~\ref{cor:density-expansion}, the diagonal value of the transition density admits the following short-time asymptotic expansion
\begin{equation}
    p(\tau,y,y)
    =
    \frac{\tau^{-1/2}}{\sigma(y)\sqrt{2\pi}}
    + r_1(\tau,y,y),
    \qquad \tau \to 0^{+},
\end{equation}
where the continuous remainder term fulfills $r_1(\tau,y,y) = O_K(\tau^{1/2})$. Multiplying \eqref{eq:VIE-first} by $\frac{\beta(\psi(y))}{\sigma(y)}$ and adding the result to \eqref{eq:VIE-second} leads to the main result of this work.
\begin{theorem}\label{th:VIE-final}
Under Assumptions~\ref{ass:smooth}--\ref{ass:func} and~\ref{ass:potential}, 
for each fixed $y$, the function $t \mapsto \partial_x f(t,y^+)$ 
is the unique solution in $C[0,T]$ of the following Volterra integral 
equation of the second kind with a continuous kernel
\begin{equation}
\alpha(t,y)\,\partial_x f(t,y^+)
=
G(t,y)
+
\int_0^t
\gamma_1(t-s,y)\,k(t,s)\,
\partial_x f(s,y^+)\,ds.
\end{equation}
Here,
\begin{equation}
G(t,y)
=
\partial_x g(t,y)
+
\frac{\beta(\psi(y))}{\sigma(y)}
\big(g(t,y)-f(t,y)\big),
\qquad
\gamma_1(\tau,y)
=
r_0(\tau,y,y)
+
\frac{\beta(\psi(y))}{\sigma(y)}
r_1(\tau,y,y),
\end{equation}
where $r_1(\tau,y,y)$ and $r_0(\tau,y,y)$ are continuous functions exhibiting $O_K(\tau^{1/2})$ behavior uniformly for $y$ in compact sets.
\end{theorem}
\begin{proof}
As before, the uniqueness of the continuous solution stems from Theorem~1.2.3 in~\cite{brunner2017}.
\end{proof}
\begin{remark} \label{rem:diff}
 As is evident from Theorems~\ref{th:VIE-second} and~\ref{th:VIE-final}, the resulting kernels involve the remainder terms $r_0$ and $r_1$, which are specified only implicitly through their asymptotic expansions. Consequently, these representations are not intended for direct computational use. Rather, they should be understood as a methodological framework: after differentiating the original problem, employing the relevant asymptotic expansions, and passing to the limit, one obtains the Volterra integral equation stated above.   
\end{remark}

\section{Applications} \label{sec:App}
At the end, we present an example that illustrates the applications of the new expression. We consider the Black--Scholes model, in which the asset price $X_t$ follows a geometric Brownian motion with constant volatility $\sigma$ and drift $r$ under the risk-neutral measure $\mathbb{Q}$. We are interested in a down-and-out barrier option, which is one of the simplest non-vanilla derivatives. Its price at time $t$, with maturity $T$, is given by
\begin{equation}
\label{eq:payoff}
f(t,x)
=
\mathbb{E}^{\mathbb{Q}}\!
\left[
(X_T-K)^+\,
\mathbf{1}_{\left\{m_t>L\right\}}
\,\middle|\, X_t=x
\right],
\qquad
m_t=\min_{t\leq s\leq T}X_s,
\end{equation}
where $K$ is the strike price and $L$ is the lower barrier (corresponding to our boundary $y$). As shown in
\cite{detemple2026}, Theorem~2.1, the target function $f(t,x)$ satisfies the following equation of the form \eqref{eq:target}:
\begin{equation}
f(t,x)
=
C_{BS}(t,x)
-\frac{\sigma L}{2}
\int_t^T
e^{-r(s-t)}
\varphi\!\left(d^-(s-t,x,L)\right)
\frac{\Delta_L(s)}{\sqrt{s-t}}
\,ds,
\end{equation}
where $C_{BS}(t,x)$ is the price of a standard European call option, defined as
\begin{equation}
C_{BS}(t,x)
=
x \, \Phi\!\left(d^{+}(T - t, x, K)\right)
-
K e^{-r(T - t)}
\Phi\!\left(d^{-}(T - t, x, K)\right),
\end{equation}
with
\begin{equation}
d^\pm(\tau,x,K)
=
\frac{1}{\sigma\sqrt{\tau}}
\left(
\log\frac{x}{K}
+
\left(r\pm\frac{\sigma^2}{2}\right)\tau
\right).
\end{equation}
Here, $\Delta_L(t)=\partial_x f(t,L^+)$ denotes the so-called option Delta at the barrier, while $\Phi(\cdot)$ and $\varphi(\cdot)$ denote, respectively, the cumulative
distribution function and the probability density function of the standard normal distribution. Taking into account that the derivative of $\varphi$ with respect to
$x$ is given by
\begin{equation}
\label{eq:density_derivative}
\partial_{x} \varphi\!\left(d^{-}(\tau, x, L)\right)
=
-\frac{L}{x}
\frac{\ell + \mu \tau}{\sigma \sqrt{2 \pi} \, \tau^{3/2}}
\exp\!\left(
-\frac{(\ell + \mu \tau)^{2}}{2 \sigma^{2} \tau}
\right),
\qquad
\mu = r - \frac{\sigma^{2}}{2},
\end{equation}
where $\ell=\log(x/L)$. We can therefore apply the procedure described in the previous Section~\ref{sec:VIE}. To decompose \eqref{eq:density_derivative} into two singular parts of orders $1/2$ and $3/2$, we note that the regular term is absent in this case. Then, after making the appropriate substitutions and carrying out the corresponding calculations, we obtain a Volterra integral equation of the form
\eqref{eq:VIE-second}:
\begin{equation}
\Delta_L(t)
=
\partial_x C_{BS}(t,L)
+
\frac{\mu}{\sigma \sqrt{2 \pi}}
\int_{t}^{T}
e^{r(s - t)}
\exp\!\left(
\frac{\mu^{2}(s - t)}{2 \sigma^{2}}
\right)
\frac{\Delta_{L}(s)}{\sqrt{s - t}}
\, ds.
\end{equation}
Now, using Theorem~2.2 in \cite{detemple2026}, which provides an analogue of the representation \eqref{eq:VIE-first}, and adding the two resulting expressions, we arrive at the following closed-form expression for the Delta:
\begin{equation} \label{eq:closed}
\Delta_L(t)
=
2\partial_x C_{BS}(t,L)
+
\frac{2\mu}{\sigma^2 L}
C_{BS}(t,L).
\end{equation}
This indeed coincides with the Delta of the barrier option, as can be verified using \cite[Eq.~(4)]{detemple2026}. An analogous explicit result can be obtained by replacing the geometric Brownian motion in \eqref{eq:payoff} with Brownian motion, which leads to the Bachelier model.
\subsection{Practical Considerations}
Explicit expressions like \eqref{eq:closed} are rarely available in practice; this is already the case for the Ornstein--Uhlenbeck process, for instance. It brings us back to the issue noted in Remark~\ref{rem:diff}. To bypass this, we can truncate the series remainders $r_0$ and $r_1$ to compute them approximately. The resulting error $\varepsilon > 0$ in the kernels $\gamma_0^\varepsilon$ and $\gamma_1^\varepsilon$ is not critical, because the inverse operator associated with the second-kind Volterra equation is bounded, which guarantees stability via the Fredholm alternative. Accordingly, we can write estimates for the difference between the solutions $\partial_x f$ and $\partial_x f^\varepsilon$. 

From a numerical standpoint, the approach offers significant advantages; one can refer to the monograph of Brunner~\cite{brunner2004}, including Sections~2.2 and~6.2 for second-kind equations with smooth and singular kernels. Conversely, singular first-kind equations suffer from theoretical limitations for numerical treatment (see Section~6.3), which motivates our focus away from them in the present study.

\section*{Acknowledgments}
The author thanks Dr. Yerkin Kitapbayev for sharing the example in Section \ref{sec:App} and for inspiring this work.


\end{document}